\documentclass{amsart}%
\usepackage{amsfonts}
\usepackage{amsmath}
\usepackage{amssymb}
\usepackage{graphicx}%
\providecommand{\U}[1]{\protect\rule{.1in}{.1in}}
\newtheorem{theorem}{Theorem}
\theoremstyle{plain}

\newtheorem{corollary}{Corollary}

\newtheorem{definition}{Definition}
\newtheorem{example}{Example}

\newtheorem{lemma}{Lemma}

\newtheorem{proposition}{Proposition}

\numberwithin{equation}{section}
\begin{document}
\title[{\normalsize On Regular Fusible Modules}]{{\normalsize On Regular Fusible Modules }}
\author{Osama A. Naji}
\address{Department of Mathematics, Sakarya University, Sakarya, Turkey.}
\email{osamanaji@sakarya.edu.tr}
\author{Mehmet \"{O}zen}
\address{Department of Mathematics, Sakarya University, Sakarya, Turkey.}
\email{ozen@sakarya.edu.tr}
\author{\"{U}nsal Tekir}
\address{Department of Mathematics, Marmara University, Istanbul, Turkey.}
\email{utekir@marmara.edu.tr}
\author{Suat Ko\c{c}}
\address{Department of Mathematics, Istanbul Medeniyet University, Istanbul, Turkey.}
\email{suat.koc@medeniyet.edu.tr}
\subjclass[2020]{16D10, 16D80, 16U99}
\keywords{regular fusible module, fusible module, fusible ring, regular fusible ring,
regular fusible submodule}

\begin{abstract}
In this article, we introduce the notion of regular fusible modules. Let $R$
be a ring with an identity and $M$ an $R$-module. An element $0\neq m\in M$ is
said to be regular fusible if there exists $r\in R$, a non zero-divisor of
$M$, such that $mr$ can be written as the sum of a torsion element and a
torsion free element in $M$. $M$ is called regular fusible if every nonzero
element of $M$ is regular fusible. We characterize regular fusible modules in
terms of fusible modules. In addition, we show that a regular fusible module
over a right duo ring is reduced and nonsingular. Moreover, we study the
regular fusible property under Cartesian product, trivial extension ring, and
module of a fractions. Also, we characterize division rings in terms of
fusible modules.

\end{abstract}
\maketitle

\section{Introduction}

Let $R$ denote a ring with an identity. The concept of fusible elements in
rings was introduced by Ghashghaei and McGovern in their work \cite{Ghas}.
Specifically, a non-zero element $a$ of $R$ is called left (right) fusible if
there exists a left (right) zero-divisor element $z$ and a non-left
(non-right) zero-divisor element $r$ such that $a=z+r$. An element in $R$ that
exhibits both left and right fusibility is said to be fusible. Furthermore, a
ring $R$ is called left (right) fusible if each non-zero element of $R$ is
identified as left (right) fusible. Later, Ko\AA \"{Y}an and Matczuk
introduced in \cite{Kosan} the concept of regular fusible rings. A ring $R$ is
called regular left fusible ring if, for every non-zero element $a$ in $R$,
there exists a regular element $s$ in $R$ (i.e., an element that is both left
and right regular) such that the product $sa$ is left fusible. Formally,
$sa=z+r$, where $z$ is a left zero divisor and $r$ is left regular. They
showed that class of regular left fusible rings is wider than the one of left
fusible rings.

Baydar and colleagues, as outlined in \cite{Baydar}, expanded the concept of
fusibility to modules. Specifically, an $R$-module $M$ is deemed fusible when
each non-zero element of $M$ can be expressed as the sum of a torsion element
and a torsion-free element. Inspired by the study of regular fusible rings,
the present paper introduces the concept of a regular fusible module, serving
as a module-theoretic counterpart to a regular fusible ring.

For the sake of completeness, we provide the notations that will be utilized
throughout this manuscript. Let $M$ be an $R$-module and $N$ a submodule of
$M$. The right annihilator of $N$, denoted by $ann_{R}(N)$, is the right ideal
defined by $\{r \in R : Nr = 0\}$. In particular, $ann_{R}(m) = \{r \in R : mr
= 0\}$ for $m\in M$. If $ann_{R}(M)=0$, then $M$ is said to be a faithful
module. The set of all torsion elements of $M$ is defined by $T_{R}(M)= \{m\in
M: mr = 0 \text{ for some nonzero element } r\in R\}$. Note that $T_{R}(M)=
\{m\in M: ann_{R}(m)\neq0 \}$. A module $M$ is called torsion free if
$T_{R}(M)=0$, and it is called a torsion module if $T_{R}(M)=M$. $T_{R}%
^{\star}(M)$ denotes the set of all torsion free elements in $M$, that is,
$T_{R}^{\star}(M)=\{ m\in M: ann_{R}(m)=0\}$. The set of all zero divisors of
$M$, denoted by $Zd(M)$, is defined by $\{r\in R: mr=0 \text{ for some } 0\neq
m\in M \}$\cite{North}. It is noteworthy that the element $1$ always belongs
to the set $R-Zd(M)$.

In section 2, we introduce the notion of regular fusible modules. We say that
a nonzero element $m$ of an $R$-module $M$ is regular fusible if there exists
$r \in R - Zd(M)$ such that the product $mr$ is fusible, signifying that $mr$
can be expressed as $mr = x + y$, where $x\in T_{R}(M)$ and $y\in T_{R}%
^{\star}(M)$. The module $M$ is said to be regular fusible module if every
non-zero element of $M$ possesses the property of being regular fusible. We
showed in Proposition \ref{prop3} and Corollary \ref{cor12} that every regular
fusible module is nonsingular, provided that $R - Zd(M)$ is contained in the
center of the ring or $R$ is a right duo ring, that is, $Ra\subseteq aR$.
Consequently, every regular fusible module over a commutative ring is
nonsingular. In addition, we give some examples of nonsingular modules that do
not exhibit regular fusibility when the set $R - Zd(M)$ is contained in the
center of the ring. This delineation signifies that the class of regular
fusible modules is strictly smaller than the class of nonsingular modules. In
fact we have the following inclusion over commutative rings: torsion free
module $\Rightarrow$ fusible module $\Rightarrow$ regular fusible module
$\Rightarrow$ nonsingular module. Furthermore, if $R$ is an integral domain,
then the four concepts coincide (Corollary \ref{cor11}). Moreover, we proved
in Theorem \ref{tlocalization} for a faithful module $M$ over a commutative
ring $R$ and $S=R-Zd(M)$, $M$ is regular fusible if and only if $S^{-1}M$ is a
fusible $S^{-1}R$-module if and only if $S^{-1}M$ is a regular fusible
$S^{-1}R$-module. In Theorem \ref{thm3}, we showed that a regular fusible
module $M$ over a right duo ring $R$ is reduced, but the converse is not true.
This result generalizes the result in \cite[Theorem 2.33]{Baydar}. Also, we
studied the regular fusible property under the Cartesian product of modules
(Corollary \ref{cor13}) and under the trivial extension ring (Theorem
\ref{thm5}). Finally, we characterize division rings in terms of fusible
modules (Theorem \ref{division}).

\section{Regular Fusible Modules}

In parallel with the regular fusible characteristic observed in rings, this
section introduces the concept of the regular fusible property within the
context of modules, and subsequently establishes several basic properties
associated with it.

\begin{definition}
Let $M$ be an $R$-module. A nonzero element $m$ of $M$ is said to be regular
fusible if there exists $r \in R - Zd(M)$ such that $mr$ is fusible, that is,
$mr$ can be written as $mr = x + y$, where $x\in T_{R}(M)$ and $y\in
T_{R}^{\star}(M)$. The module $M$ is called regular fusible if every nonzero
element of $M$ is regular fusible.
\end{definition}

Note that $1 \in R - Zd(M)$, so, it is straightforward from the definition
that every fusible module is regular fusible. Since every element in
$T_{R}^{\star}(M)$ is fusible, it is enough to see that for every nonzero
torsion element $m$, there exists $r \in R - Zd(M)$ such that the element $mr$
has a fusible representation.

\begin{example}
\label{Ex1} (1) Torsion free modules are regular fusible, since $m\cdot1 =m =
0 + m$ for every nonzero element $m$ of $M$.\newline(2) Torsion modules can
not be regular fusible, since $T_{R}(M)=M$ and $T_{R}^{\star}(M)=\emptyset$.
\end{example}

In general, a regular fusible module may not be fusible. The following example
demonstrates that the class of fusible modules is decidedly more strict than
the class of regular fusible modules.

\begin{example}
\cite{Kosan} Let $F$ be a field and $R=F<x, y>$ with $x^{2}=0$. Consider the
$R$-module $R$. Cohn showed in \cite{Coh} that $xR$ is the set of all left
zero divisors of $R$. Hence, if $m\in R$ is a left zero divisor, then it
cannot be written as a fusible decomposition since the sum of two left zero
divisors is a zero divisor. Therefore, $R$ is not fusible. The element $y$ is
regular in $R$, and for each $0\neq m\in R$, we have $my = 0+my$ is a fusible
decomposition of $my$. Thus, $R$ is a regular fusible module.
\end{example}

\begin{example}
\label{Ex2} (1) Let $M$ be a simple module over a local ring $R$ that is not a
field. Then, for every nonzero element $m$ of $M$, $ann_{R}(m)$ is a nonzero
maximal ideal of $R$. Hence $m\in T_{R}(M)$, that is, $M$ is a torsion module.
Therefore, $M$ is not regular fusible.\newline(2) Let $M$ be a module over a
field $R$, then $ann_{R}(m)=0$ for every nonzero element $m$ of $M$. Thus,
$T_{R}(M)=0$ and $m\in T_{R}^{\star}(M)$. So, $M$ is regular fusible.
\end{example}

The converse of Example \ref{Ex1}(1) may not be true in general. Consider the
$\mathbb{Z}_{6}$-module $\mathbb{Z}_{6}$. It is clear that $\mathbb{Z}_{6}$ is
not a torsion free module. Note that $T_{R}(\mathbb{Z}_{6})=\{0, 2, 3, 4\}$
and $T_{R}^{\star}(\mathbb{Z}_{6})=\{1, 5\}$. It can be easily seen that for
every nonzero element $m$ of $\mathbb{Z}_{6}$, there exists $r \in
\mathbb{Z}_{6} - Zd(\mathbb{Z}_{6})=\{1, 5\}$ such that $mr$ is fusible.

\begin{proposition}
\label{Prop1} If $T_{R}(M)$ forms a submodule of $M$, then $M$ is torsion free
if and only if $M$ is regular fusible.
\end{proposition}

\begin{proof}
Let $m$ be a nonzero element of $T_{R}(M)$. There exists $r \in R - Zd(M)$
such that $mr = x + y$, where $x\in T_{R}(M)$ and $y\in T_{R}^{\star}(M)$.
Then $y = mr - x \in T_{R}(M)$. A contradiction.
\end{proof}

It is easy to see that if $R$ is an integral domain, then the set of all
torsion elements of $M$ forms a submodule of $M$.

\begin{corollary}
\label{cor1} Let $R$ be an integral domain. Then $M$ is torsion free if and
only if $M$ is regular fusible.
\end{corollary}

We say that a ring $R$ satisfies the comparability relation between right
annihilators of the elements of $M$ if $ann_{R}(m_{1})\subseteq ann_{R}%
(m_{2})$ or $ann_{R}(m_{2})\subseteq ann_{R}(m_{1})$ for every $m_{1}$,
$m_{2}\in M$.

\begin{proposition}
\label{prop2} Suppose that $R$ is a ring with comparability relation between
right annihilators of the elements of $M$, and $R - Zd(M)$ is contained in the
center of $R$. Then, $M$ is a regular fusible module if and only if $M$ is
torsion free.
\end{proposition}

\begin{proof}
Let $M$ be a regular fusible module, and $0\neq m \in T_{R}(M)$. Then there
exists $r \in R - Zd(M)$ such that $mr = x + y$, for some $x\in T_{R}(M)$ and
$y\in T_{R}^{\star}(M)$. Note that $ann_{R}(m)\neq0$, and this implies
$ann_{R}(mr)\neq0$. Assume that $0\neq t \in ann_{R}(mr)\subseteq ann_{R}(x)$.
Then $mrt = xt + yt$, and hence $0=yt$, and we get $y\in T_{R}(M)$. Similarly,
if $0\neq t\in ann_{R}(x)\subseteq ann_{R}(mr)$, we obtain $y\in T_{R}(M)$. A contradiction.
\end{proof}

An ideal $I$ of a ring $R$ is called essential, denoted by $I \leq_{e} R$, if
for every nonzero ideal $J$ of $R$, $I\cap J \neq0$. Let $M$ be an $R$-module.
$Z(M)=\{m\in M: ann_{R}(m)\leq_{e} R\}$ is called the singular submodule of
$M$. A module $M$ is said to be singular when $Z(M)$ coincides with $M$. If
$Z(M)$ equals zero, the module $M$ is referred to as a nonsingular module.

\begin{proposition}
\label{prop3} Suppose that $R - Zd(M)$ is contained in the center of the ring
$R$. If $M$ is a regular fusible module, then $M$ is nonsingular.
\end{proposition}

\begin{proof}
Let $m$ be a nonzero element of $Z(M)$. Since $M$ is regular fusible, there
exists $r\in R-Zd(M)$, $x\in T_{R}(M)$, and $y\in T_{R}^{\star}(M)$ such that
$mr=x+y$. Note that $mr\in Z(M)$. Indeed, let $J$ be a nonzero ideal of $R$,
then we obtain $0\neq ann_{R}(m)\cap J\subseteq ann_{R}(mr)\cap J$. Now, since
$ann_{R}(x)\neq0$, we have $ann_{R}(mr)\cap ann_{R}(x)\neq0$. Let $0\neq w\in
ann_{R}(mr)\cap ann_{R}(x)\subseteq ann_{R}(mr-x)=ann_{R}(y)=0$. This is a
contradiction, thus, $Z(M)=0$.
\end{proof}

\begin{corollary}
\label{cor12} (i) Every regular fusible module over a right duo ring is
nonsingular.\newline(ii) Every regular fusible module over a commutative ring
is nonsingular.
\end{corollary}

\begin{proof}
(i): Let $M$ be a regular fusible module over a right duo ring $R$. Choose
$r\in R$ and $m\in M$. Let $b\in ann_{R}(m)$. Then $mb=0$. Since $R$ is right
duo ring, $rb=br^{\prime}$ for some $r^{\prime}\in R$. Then we have
$mrb=mbr^{\prime}=0$, that is, $ann_{R}(m)\subseteq ann_{R}(mr)$. By a similar
argument in the proof of Proposition \ref{prop3}, one can show that $M$ is
nonsingular if $M$ is a regular fusible module.\newline(ii): Follows from (i).
\end{proof}

By Proposition \ref{prop3}, if there exists a nonzero element $m$ of $M$ such
that $ann_{R}(m)\cap I \neq0$ for every nonzero ideal $I$ of $R$, then $M$ is
not regular fusible, provided that $R-Zd(M)$ is contained in the center of
$R$. Let $M=R=$ $%
\begin{bmatrix}
\mathbb{Z}_{2} & 0 & 0\\
0 & \mathbb{Z}_{2} & 0\\
\mathbb{Z}_{2} & \mathbb{Z}_{2} & \mathbb{Z}_{2}%
\end{bmatrix}
$ and consider the $R$-module $R$. Let $0\neq m=$$%
\begin{bmatrix}
0 & 0 & 0\\
0 & 0 & 0\\
1 & 0 & 0
\end{bmatrix}
$. Note that $ann_{R}(m)\cap I\neq0$ for every nonzero ideal $I$ of $R$. Hence
$m\in Z(M)$ and $M$ is not nonsingular. Thus, $M$ is not a regular module.
Indeed, the element $m$ mentioned above is not regular fusible. Since for each
$r\in R-Zd(M)=T_{R}^{\star}(M)=\left\lbrace
\begin{bmatrix}
1 & 0 & 0\\
0 & 1 & 0\\
0 & 0 & 1
\end{bmatrix}
,
\begin{bmatrix}
1 & 0 & 0\\
0 & 1 & 0\\
1 & 0 & 1
\end{bmatrix}
,
\begin{bmatrix}
1 & 0 & 0\\
0 & 1 & 0\\
1 & 1 & 1
\end{bmatrix}
,
\begin{bmatrix}
1 & 0 & 0\\
0 & 1 & 0\\
0 & 1 & 1
\end{bmatrix}
\right\rbrace $, $mr=%
\begin{bmatrix}
0 & 0 & 0\\
0 & 0 & 0\\
1 & 0 & 0
\end{bmatrix}
$ does not have a fusible decomposition $mr=x+y$, where $x\in T_{R}%
(M)=R-T_{R}^{\star}(M)$ and $y\in T_{R}^{\star}(M)$.\newline The converse of
Proposition \ref{prop3} is not true. Let $R=$ $%
\begin{bmatrix}
\mathbb{Z}_{2} & 0\\
\mathbb{Z}_{2} & \mathbb{Z}_{2}%
\end{bmatrix}
$ and consider the $R$-module $R$. Note that $R-Zd(R)=T_{R}^{\star
}(R)=\left\lbrace
\begin{bmatrix}
1 & 0\\
0 & 1
\end{bmatrix}
,
\begin{bmatrix}
1 & 0\\
1 & 1
\end{bmatrix}
\right\rbrace $ and $T_{R}(R)=R-T_{R}^{\star}(R)$. $R$ is not regular fusible
since $m=%
\begin{bmatrix}
0 & 0\\
1 & 0
\end{bmatrix}
$ is not a regular fusible element. Observe that $mr$ can not be written as a
fusible decomposition for every element $r\in R-Zd(R)$. It is easy to check
that the set $Z(R)=\{a \in R: ann_{R}(a)\leq_{e} R\}=0$. Therefore, the ring
$R$ is nonsingular. The following result shows that the converse of
Proposition \ref{prop3} is true when the ring $R$ is an integral domain.

\begin{proposition}
\label{prop4} Let $R$ be an integral domain and $M$ an $R$-module. If $M$ is
nonsingular, then $M$ is fusible.
\end{proposition}

\begin{proof}
Let $m$ be a nonzero element of $M$. Under the assumption that "$M$ is
nonsingular", we prove that $ann_{R}(m)=0$. Assume that $ann_{R}(m)\neq0$.
Since $m\notin Z(M)$, we have $ann_{R}(m)\nleq_{e} R$. Then there exists a
nonzero ideal $J$ of $R$ such that $ann_{R}(m)\cap J=0$. Choose $0\neq x\in J$
and $y\in ann_{R}(m)$. Since $R$ is a domain, $0\neq xy\in ann_{R}(m)\cap J$
which is a contradiction. Thus we have $ann_{R}(m)=0$, that is, $M$ is a
torsion-free module. Thus by Example \ref{Ex1}, $M$ is fusible.
\end{proof}

Observe that in the previous example, $R=$ $%
\begin{bmatrix}
\mathbb{Z}_{2} & 0\\
\mathbb{Z}_{2} & \mathbb{Z}_{2}%
\end{bmatrix}
$ is not an integral domain and the properties of nonsingularity and
fusibility do not coincide. In the following corollary, we give a
characterization of regular fusible modules in terms of nonsingular, torsion
free, and fusible modules, provided that the ring $R$ is an integral domain.

\begin{corollary}
\label{cor11} Let $R$ be an integral domain and $M$ an $R$-module. The
following statements are equivalent:\newline(i) $M$ is fusible.\newline(ii)
$M$ is regular fusible.\newline(iii) $M$ is nonsingular.\newline(iv) $M$ is
torsion free.
\end{corollary}

\begin{proof}
(i)$\Rightarrow$(ii): It is evident.\newline(ii)$\Rightarrow$(iii): It follows
from Proposition \ref{prop3}.\newline(iii)$\Rightarrow$(i): It follows from
Proposition \ref{prop4}.\newline(ii)$\Leftrightarrow$ (iv): It follows from
Corollary \ref{cor1}.\newline
\end{proof}

A subset $S$ of a ring $R$ is called a multiplicatively closed set if $1\in S$
and $ab\in S$ for all $a, b\in S$. Observe that if $0\in S$, then the ring of
fractions $S^{-1}R$ is the zero ring. Therefore, we assume that $0\notin S$.
Next, we investigate the regular fusible property within the context of the
module of fractions. To begin, we present the following lemma.

\begin{lemma}
\label{lem1} Let $R$ be a commutative ring and $S$ a multiplicatively closed
subset of $R$ consisting of regular elements. For every $m\in M$ and $s\in S$,
$m\in T_{R}(M)$ if and only if $\frac{m}{s}\in T_{S^{-1}R}(S^{-1}M)$.
\end{lemma}

\begin{proof}
First, it is easy to see that $ann_{S^{-1}R}(\frac{m}{s})=S^{-1}[ann_{R}(m)]$
for every $m\in M$ and $s\in S$. If $m\notin T_{R}(M)$, then $ann_{R}(m)=0$
and so $ann_{S^{-1}R}(\frac{m}{s})=S^{-1}[ann_{R}(m)]=0$ which implies that
$\frac{m}{s}\notin T_{S^{-1}R}(S^{-1}M)$. Now, assume that $m\in T_{R}(M)$.
Then there exists $0\neq c\in R$ such that $mc=0$. This implies that $\frac
{m}{s}\frac{c}{1}=0$ and $\frac{c}{1}\neq0$, that is, $\frac{m}{s}\in
T_{S^{-1}R}(S^{-1}M)$.
\end{proof}

\begin{theorem}
\label{th2} Let $R$ be a commutative ring and $S$ be a multiplicatively closed
subset of $R$. Then the following statements are satisfied:\newline(i) If $M$
is a regular fusible $R$-module and $S$ consists of regular elements, then so
is $S^{-1}R$-module $S^{-1}M$.\newline(ii) Suppose that $S\subseteq R-Zd(M)$
and $M$ is a faithful $R$-module. If $S^{-1}M$ is a regular fusible $S^{-1}%
R$-module, then $M$ is a regular fusible $R$-module.
\end{theorem}

\begin{proof}
(i) Let $\frac{0}{1}\neq\frac{m}{s}\in S^{-1}M$. Since $0\neq m\in M$, there
exists $r\in R-Zd(M)$ such that $mr=x+y$ for some $x\in T_{R}(M)$ and $y\in
T_{R}^{\star}(M)$. By Lemma \ref{lem1}, we have $\frac{x}{s}\in T_{S^{-1}%
R}(S^{-1}M)$ and $\frac{y}{s}\in T_{S^{-1}R}^{\star}(S^{-1}M)$. Now we claim
that $\frac{r}{1}\in S^{-1}R-Zd(S^{-1}M)$. Indeed, if $\frac{r}{1}\in
Zd(S^{-1}M)$, then there exists $\frac{0}{1}\neq\frac{n}{v}\in S^{-1}M$ such
that $\frac{n}{v}\cdot\frac{r}{1}=\frac{nr}{v} = \frac{0}{1}$. Hence, $nru =
(nu)r=0$ for some $u\in S$. Note that $nu\neq0$, so we obtain that $r\in
Zd(M)$, a contradiction. Therefore, $\frac{m}{s}\cdot\frac{r}{1}=\frac{x}%
{s}+\frac{y}{s}$ is a fusible decomposition of $\frac{m}{s}\cdot\frac{r}{1}$
and so $S^{-1}M$ is regular fusible.\newline(ii) First note that since $M$ is
a faithful module, every element of $R-Zd(M)$ is a regular element of $R$
which implies that $S$ consists of regular elements. Let $0\neq m\in M$. Since
$S\subseteq R-Zd(M)$, it is clear that $\frac{m}{1}$ is a nonzero element of
$S^{-1}M$. Then by assumption, there exist $\frac{x}{s_{1}}\in T_{S^{-1}%
R}(S^{-1}M)$, $\frac{y}{s_{2}}\in T_{S^{-1}R}^{\star}(S^{-1}M)$ and $\frac
{r}{s}\in S^{-1}R-Zd(S^{-1}M)$ such that $\frac{m}{1}\cdot\frac{r}{s}=\frac
{x}{s_{1}}+\frac{y}{s_{2}}$. Hence, $\frac{mr}{s}=\frac{xs_{2}+ys_{1}}%
{s_{1}s_{2}}$, and so $mrs_{1}s_{2}=xs_{2}s+ys_{1}s$ since $S\subseteq
R-Zd(M)$. Put $rs_{1}s_{2}=u$, $s_{2}s=t$, $s_{1}s=v$ and note that $u,t,v\in
R-Zd(M)$. Also it is easy to see that $ann_{R}(xt)=ann_{R}(x)$ and
$ann_{R}(yv)=ann_{R}(y)$. Then by Lemma \ref{lem1}, we have $xt\in T_{R}(M)$
and $yv\in T_{R}^{\star}(M)$. This gives $mu=xt+yv$ is a regular fusible
decomposition and so $M$ is a regular fusible $R$-module.
\end{proof}

Let $R$ be a commutative ring, $M$ an $R$-module and $S=R-Zd(M)$. Then it is
clear that $S$ is a multiplicatively closed subset of $R$. If $M$ is a
faithful $R$-module, then $S$ consists of regular elements of $R$. For
$S=R-Zd(M)$, the localization $S^{-1}M$ is an $S^{-1}R$-module and it is
called total quotient module. Here, we denote the total quotient module by
$q(R)$-module $q(M)$.

\begin{theorem}
\label{tlocalization} Suppose that $M$ is a faithful module over a commutative
ring $R$. The following statements are equivalent:\newline(i) $M$ is a regular
fusible $R$-module. (ii) $q(M)$ is a fusible $q(R)$-module. (iii) $q(M)$ is a
regular fusible $q(R)$-module
\end{theorem}

\begin{proof}
(i)$\Rightarrow$(ii): Let $0\neq\frac{m}{s}\in q(M)$, where $m\in M$ and $s\in
R-Zd(M)$. Then $0\neq m\in M$. As $M$ is a regular fusible module, there exist
$r\in R-Zd(M)$, $x\in T_{R}(M)$ and $y\in T_{R}^{\star}(M)$ such that
$mr=x+y$. This gives $\frac{m}{s}=\frac{mr}{sr}=\frac{x}{sr}+\frac{y}{sr}$.
Since $M$ is faithful, it is clear that $S=R-Zd(M)$ consists of regular
elements of $R$. Then by Lemma \ref{lem1}, we have $\frac{x}{sr}\in
T_{S^{-1}R}(S^{-1}M)$ and $\frac{y}{sr}\in T_{S^{-1}R}^{\star}(S^{-1}M)$. Thus
$\frac{m}{s}=\frac{x}{sr}+\frac{y}{sr}$ is a fusible element of $S^{-1}M$.
Hence, $q(M)$ is a fusible $q(R)$-module.\newline(ii)$\Rightarrow$(iii): It is
clear.\newline(iii)$\Rightarrow$(i): Follows from Theorem \ref{th2}.\newline
\end{proof}

An $R$-module $M$ is said to be reduced if for every $m\in M$ and every $r\in
R$ with $mr=0$, then $mR\cap Mr=0$ \cite{Lee}. Observe that a ring $R$ is
reduced if and only if it is a reduced module over itself. In the following,
we show that every regular fusible module over a right duo ring is reduced.

\begin{theorem}
\label{treduced} Every regular fusible module over a right duo ring is reduced.
\end{theorem}

\begin{proof}
Let $M$ be a regular fusible module over a right duo ring $R$. Choose $a\in R$
and $m\in M$ such that $ma=0$. Now, we will show that $mR\cap Ma=0$. To see
this take $x=mr=m^{\star}a$ for some $r\in R$ and $m^{\star}\in M$. Then we
have $mra=m^{\star}a^{2}$. Since $R$ is right duo, there exists $y\in R$ such
that $ra=ay$ and this gives $mra=m^{\star}a^{2}=may=0$. Now we will show that
$m^{\star}a=0$. Suppose that $m^{\star}a\neq0$. Since $M$ is regular fusible
module, there exist $u\in R-Zd(M)$, $z\in T_{R}(M)$ and $t\in T_{R}^{\star
}(M)$ such that $m^{\star}au=z+t$. Since $z\in T_{R}(M)$, there exists $0\neq
c\in R$ such that $zc=0$. Multiply $m^{\star}au=z+t$ by $ac$, we obtain
$m^{\star}auac=zac+tac$. Since $R$ is right duo, we have $zac=zcd$ for some
$d\in R$ and so $zac=0$. Similarly, $R$ right duo gives $ua=au^{\prime}$ for
some $u^{\prime}\in R$. Thus we have $m^{\star}auac=m^{\star}a^{2}u^{\prime
}c=0$. Then we get $tac=0$. Since $t\in T_{R}^{\star}(M)$, we get $ac=0$.
Again multiply $m^{\star}au=z+t$ by $c$, we have $m^{\star}auc=zc+tc=tc$.
Since $R$ is right duo, there exists $c^{\prime}\in R$ such that
$uc=cc^{\prime}$ and so we have $m^{\star}auc=m^{\star}acc^{\prime}=tc=0$. As
$t\in T_{R}^{\star}(M)$, we obtain $c=0$ which is a contradiction. Thus
$x=mr=m^{\star}a=0$, that is, $M$ is reduced.
\end{proof}

In \cite{Baydar}, Baydar et al. showed that if $M$ is a fusible module over a
duo ring $R$, then $M$ is reduced. However, in the following theorem, the
claim is still true if we replace duo ring by right duo.

\begin{theorem}
\label{thm3} (i) Let $M$ be a fusible module over a right duo ring $R$. Then
$M$ is reduced.\newline(ii) A regular fusible module over a commutative ring
is reduced.
\end{theorem}

The converse of Theorem \ref{thm3} (ii) is not true. See the following example.

\begin{example}
Consider the $\mathbb{Z}_{8}$-module $4 \mathbb{Z}_{8}=\{0, 4\}$. It is easy
to check that the module $4\mathbb{Z}_{8}$ is reduced \cite{Lee}. Note that
$\mathbb{Z}_{8}-Zd(4\mathbb{Z}_{8})=\{1, 3, 5, 7\}$, $T_{\mathbb{Z}_{8}%
}(4\mathbb{Z}_{8})=\{0, 4\}$, and $T_{\mathbb{Z}_{8}}^{\star}(4\mathbb{Z}%
_{8})=\emptyset$. Thus, $4\mathbb{Z}_{8}$ is not regular fusible.
\end{example}

Let $M_{i}$ be an $R_{i}$-module for every $i=1, 2, \cdots, n$. Assume that
$R={\displaystyle \prod_{i=1}^{n} R_{i}}$, $M={\displaystyle \prod_{i=1}^{n}
M_{i}}$ and $M$ an $R$-module. Recall from \cite{Baydar} that $T_{R}^{\star
}\left(  {\displaystyle \prod_{i=1}^{n} M_{i}} \right)  = \left(
{\displaystyle \prod_{i=1}^{n} T_{R_{i}}^{\star}(M_{i})} \right)  $. In the
following, we study the regular fusible property in Cartesian product of modules.

\begin{theorem}
\label{Th6} Let $M_{i}$ be an $R_{i}$-module for $i=1, 2$ and $R=R_{1}\times
R_{2}$ and $M=M_{1}\times M_{2}$. Then the following are equivalent:\newline%
(i) $M$ is a regular fusible $R$-module.\newline(ii) $M_{i}$ is a regular
fusible $R_{i}$-module for every $i=1, 2$.
\end{theorem}

\begin{proof}
(i)$\Rightarrow$(ii): Let $m_{1}$ be a nonzero element of $M_{1}$. Since
$0\neq(m_{1}, 0)\in M$ and $M$ is regular fusible, there exist $(r_{1},
r_{2})\in R-Zd(M)$, $(x_{1}, x_{2})\in T_{R}(M)$ and $(y_{1}, y_{2})\in
T_{R}^{\star}(M)$ such that $(m_{1}, 0)(r_{1}, r_{2})=(x_{1}, x_{2})+(y_{1},
y_{2})$. Thus, $m_{1}r_{1}=x_{1}+y_{1}$. Note that $y_{1}\in T_{R_{1}}^{\star
}(M_{1})$ and $y_{2}\in T_{R_{2}}^{\star}(M_{2})$. Since $(x_{1}, x_{2})\in
T_{R}(M)$, there exists $0\neq(a_{1}, a_{2})\in R$ such that $(x_{1},
x_{2})(a_{1}, a_{2})=(0, 0)$. As $0=x_{2}+y_{2}$, $0=x_{2}a_{2}+y_{2}a_{2}$
and hence $0=y_{2}a_{2}$. As $y_{2}\in T_{R_{2}}^{\star}(M_{2})$, we obtain
that $a_{2}=0$, and so $a_{1}\neq0$. Thus, $x_{1}\in T_{R_{1}}(M_{1})$. Now,
assume that $r_{1}\in Zd(M_{1})$. Then there exists $0\neq n\in M_{1}$ such
that $nr_{1}=0$. Then, we get $(n, 0)(r_{1}, r_{2})=(0, 0)$, which means that
$(r_{1}, r_{2})\in Zd(M)$, a contradiction. Therefore, $m_{1}r_{1}=x_{1}%
+y_{1}$ is a fusible decomposition of $m_{1}r_{1}$, and thus $M_{1}$ is
regular fusible. By applying the same argument, $M_{2}$ is regular
fusible.\newline(ii)$\Rightarrow$(i):Suppose that $M_{i}$ is regular fusible
for each $i=1, 2$. Let $0\neq(m_{1}, m_{2})\in M$. Then for every $i=1, 2$,
there exist $x_{i}\in T_{R_{i}}(M_{i})$, $y_{i}\in T_{R_{i}}^{\star}(M_{i})$
and $r_{i}\notin Zd(M_{i})$ such that $m_{i}r_{i}=x_{i}+y_{i}$. Then for every
$x_{i}$, there exists $0\neq t_{i}\in R_{i}$ such that $x_{i}t_{i}=0$. Now, if
$m_{i}\neq0$, define $w_{i}=x_{i}$ and $z_{i}=y_{i}$. If $m_{i}=0$, let $0\neq
n_{i}\in M_{i}$, so $n_{i}r_{i}^{\prime}$ has a fusible decomposition for some
$r_{i}^{\prime}\notin Zd(M_{i})$. Thus, $n_{i}r_{i}^{\prime}=f_{i}+g_{i}$, for
some $f_{i}\in T_{R_{i}}(M_{i})$ and $g_{i}\in T_{R_{i}^{\star}}(M_{i})$.
Hence, when $m_{i}=0$, define $w_{i}=-g_{i}$ and $z_{i}=g_{i}$. So, we get
$(m_{1}, m_{2})(r_{1}, r_{2})=(w_{1}, w_{2})+(z_{1}, z_{2})$. Note that
$(z_{1}, z_{2})\in T_{R}^{\star}(M)$. To show that $(w_{1}, w_{2})\in
T_{R}(M)$, let $a_{i}= t_{i}$ if $m_{i}\neq0$ and $a_{i}=0$ if $m_{i}=0$.
Then, we have $(w_{1}, w_{2})(a_{1}, a_{2})=(0, 0)$. Now, if $(r_{1},
r_{2})\in Zd(M)$, then there exists a nonzero element $(d_{1}, d_{2})\in M$
such that $(d_{1}, d_{2})(r_{1}, r_{2})=(0, 0)$. So, $d_{1}r_{1}=0$ and
$d_{2}r_{2}=0$. Since $(d_{1}, d_{2})$ is nonzero, $d_{i}\neq0$ for some
$i\in\{1, 2\}$. Thus, $r_{i}\in Zd(M_{i})$, which is a contradiction.
Therefore, $(m_{1}, m_{2})$ is regular fusible and $M$ is a regular fusible module.
\end{proof}

\begin{corollary}
\label{cor13} Let $M_{i}$ be an $R_{i}$-module for each $i =1, 2, \cdots, n$.
Suppose that $R = R_{1}\times R_{2} \times\cdots\times R_{n}$ and $M =
M_{1}\times M_{2}\times\cdots\times M_{n}$. The following statements are
equivalent:\newline(i) $M$ is a regular fusible $R$-module.\newline(ii)
$M_{i}$ is a regular fusible $R_{i}$-module for every $i=1, 2, \cdots, n$.
\end{corollary}

A submodule $N$ of $R$-module $M$ is is said to be regular fusible, if each
nonzero element of $N$ is regular fusible in $M$. It is clear that if $M$ is
regular fusible, then every submodule of $M$ is regular fusible. Moreover, non
regular fusible module may admit a regular fusible submodule. For instance,
Let $R=\mathbb{Z} \times\mathbb{Z}_{4}$ and consider the $R$-module $R$. Since
$\mathbb{Z}_{4}$ is not regular fusible, then by Theorem \ref{Th6}, $R$ is not
regular fusible. However, $R$ has a regular fusible submodule, it easy to
check that $\mathbb{Z}\times\{0\}$ is are regular fusible submodule of
$R$.\newline Note that the intersection of regular fusible submodules is also
regular fusible. In addition, the intersection of submodules that are not
possessing the regular fusible property may produce a regular fusible
submodule. For example, let $R=\mathbb{Z}_{4}\times\mathbb{Z}_{18}$ and
consider the $R$-module $R$. $6$ and $12$ are not regular fusible elements in
$\mathbb{Z}_{18}$ and $2$ is not regular fusible in $\mathbb{Z}_{4}$. Observe
that the submodule $\{0\}\times\{0, 9\}=\{0\}\times\{0, 3, 6, 9, 12, 15\}
\cap\{0, 2\}\times\{0, 9\}$ is regular fusible, but neither $\{0\}\times\{0,
3, 6, 9, 12, 15\}$ nor $\{0, 2\}\times\{0, 9\}$ is regular fusible.\newline

Let $M$ be an $A$-bimodule and consider the trivial extension $A\propto
M=\{(a,m) | a\in A, m\in M\}$. Then $A\propto M$ is a ring with identity
$(1,0)$ according to componentwise addition and following multiplication:
$(a,m)(b,m^{\prime})=(ab,am^{\prime}+mb)$ for every $a,b\in A$ and
$m,m^{\prime}\in M$. For every $A$-bimodule $M$, we denote the right (left)
annilator of $M$ by $ann_{r}(M)$ ($ann_{l}(M)$). Now, we investigate the left
zero divisors of $A\propto M$ with the following lemma. In the following
lemma, we denote the set of all left zero divisors of a ring $A$ by
$zd_{l}(A)$, the set of all left zero divisors of an $A$-bimodule $M$ by
$Zd_{l}(M)$. Also, for every $a\in A$, the right annihilator of $a$ in $A$ is
denoted by $ann_{r}(a)$ and $ann_{RM}(a)$ is the set of $\{m\in M: am=0\}$.

\begin{lemma}
\label{lem2} Let $M$ be an $A$-bimodule, $A\propto M$ be the trivial extension
and $ann_{l}(M)\subseteq ann_{r}(M)$. For every $a\in A$ and $m\in M$, the
following statements are satisfied:\newline(i) $(a,m)\in zd_{l}(A\propto M)$
if and only if $a\in zd_{l}(A)\cup Zd_{l}(M)$.\newline(ii) $ann_{r}%
(a,m)=(0,0)$ if and only if $ann_{r}(a)=0$ and $ann_{RM}(a)=0_{M}$.
\end{lemma}

\begin{proof}
(i): Let $(a,m)$ be a left zero divisor of $A\propto M$. Then there exists
$(0,0)\neq(x,m^{\prime})\in A\propto M$ such that $(a,m)(x,m^{\prime
})=(ax,am^{\prime}+mx)=(0,0)$. If $x=0$, then we have $am^{\prime}=0$ which
implies that $a\in Zd_{l}(M)$. If $x\neq0$, then $ax=0$ gives $a\in zd_{l}%
(A)$. For the converse, first assume that $a\in zd_{l}(A)$. Then there exists
$0\neq x\in A$ such that $ax=0$. If $x\in ann_{r}(M)$, then we conclude that
$(a,m)(x,0)=(ax,mx)=(0,0)$. Thus, $(a,m)\in zd_{l}(A\propto M)$. Now, assume
that $x\notin ann_{r}(M)$. Since $ann_{l}(M)\subseteq ann_{r}(M)$, we have
$x\notin ann_{l}(M)$. Then there exits $0\neq m^{\prime}$ such that
$xm^{\prime}\neq0$. This gives $(a,m)(0,xm^{\prime})=(0,axm^{\prime})=(0,0)$
and so $(a,m)$ is a left zero divisor of $A\propto M$. Now, assume that $a\in
Zd_{l}(M)$. Then $am^{\prime}=0$ for some $0\neq m^{\prime}\in M$. Thus we
conclude that $(a,m)(0,m^{\prime})=(0,0)$ and so $(a,m)$ is a left zero
divisor of $A\propto M$.\newline(ii): Follows from (i).
\end{proof}

\begin{theorem}
\label{thm5} Let $A$ be a ring and $M$ be an $A$-bimodule such that
$ann_{l}(M)\subseteq ann_{r}(M)$. Then $A\propto M$ is a left fusible ring if
and only if $A$ is a left fusible ring and $M=0$.
\end{theorem}

\begin{proof}
The "only if" part follows from the fact that $A\propto M=A\propto0\cong A$
provided that $M=0$. Now, suppose that $A\propto M$ is a left fusible ring.
Now, we will show that $M=0$. Choose $0\neq m\in M$. Since $(0,0)\neq(0,m)\in
A\propto M$, there exist $(a,m^{\prime})\in zd_{l}(A\propto M)$ and
$(b,m^{\star})\notin zd_{l}(A\propto M)$ such that $(0,m)=(a,m^{\prime
})+(b,m^{\star})$. This gives $b=-a$ and $m^{\star}=m-m^{\prime}$. Since
$(b,m^{\star})=(-a,m-m^{\prime})\notin zd_{l}(A\propto M)$, by Lemma
\ref{lem2}, we have $ann_{r}(-a)=0$ and $ann_{RM}(a)=0_{M}$. On the other
hand, since $(a,m)$ is a left zero divisor of $A\propto M$, we have $a\in
zd_{l}(A)\cup Zd_{l}(M)$. This gives a contradiction. Thus $M=0$ and the rest
follows from the isomorphism $A\propto M\cong A$
\end{proof}

\begin{theorem}
Let $M$ be an $A$-bimodule and $ann_{l}(M)\subseteq ann_{r}(M)$. If $A\propto
M$ is a regular left fusible ring and $Zd_{l}(M)\subseteq zd_{l}(A)$, then $A$
is a regular left fusible ring and $M=0$.
\end{theorem}

\begin{proof}
Let $0\neq a\in A$. As $(0,0)\neq(a,0)\in A\propto M$ and $A\propto M$ is a
regular left fusible ring, there exists a regular element $(x,m)\in A\propto
M$ such that $(x,m)(a,0)=(xa,ma)$ is a left fusible element. Then there exist
$(z,m^{\prime})\in zd_{l}(A\propto M)$ and $(t,m^{\star})\notin zd_{l}%
(A\propto M)$ such that $(x,m)(a,0)=(z,m^{\prime})+(t,m^{\star})$. This gives
$xa=z+t$ where $t\notin zd_{l}(A)$ by Lemma \ref{lem2}. Since $(z,m^{\prime
})\in zd_{l}(A\propto M)$ and $Zd_{l}(M)\subseteq zd_{l}(A)$, again by Lemma
\ref{lem2}, we have $z\in zd_{l}(A)$. Thus $xa=z+t$ is a left fusible
decomposition of $xa$. Now, choose $0\neq m\in M$. Since $(0,0)\neq(0,m)\in
A\propto M$ and $A\propto M$ is a regular left fusible ring, there exists a
regular element $(x,m^{\prime})\in A\propto M$ such that $(x,m^{\prime
})(0,m)=(0,xm^{\prime})=(y,m^{\star})+(-y,xm^{\prime}-m^{\star})$ for some
$(y,m^{\star})\in zd_{l}(A\propto M)$ and $(-y,xm^{\prime}-m^{\star})\notin
zd_{l}(A\propto M)$. By Lemma \ref{lem2}, $y\in zd_{l}(A)$ and $-y\notin
zd_{l}(A)$ which is a contradiction. Thus, $M=0$.
\end{proof}

\begin{corollary}
Let $M$ be an $A$-bimodule, $ann_{l}(M)\subseteq ann_{r}(M)$ and
$Zd_{l}(M)\subseteq zd_{l}(A)$. Then $A\propto M$ is a regular left fusible
ring if and only if $A$ is a regular left fusible ring and $M=0$.
\end{corollary}

We point out that $A\propto M$ can not be a regular fusible ring unless $M=0$.
Let $M$=$R$=$\mathbb{Z}$. Then even though the ring $\mathbb{Z}$ is regular
fusible and $\mathbb{Z}$-module $\mathbb{Z}$ is torsion free, the idealization
$\mathbb{Z} \propto\mathbb{Z}$ is not regular fusible. Indeed, the element
$(0,1)$ is not regular fusible. If it were, $(0, 1)$ would possess a regular
fusible decomposition $(0,1)(r, m)=(z, x)+(t, y)$, where $(r, m)$ and $(t,y)$
are regular elements and $(z, x)$ is a zero divisor. Then we would have
$0=z+t$ and $t$ is a zero divisor in $\mathbb{Z}$. Thus by Lemma \ref{lem2},
$(t, y)$ would be a zero divisor, a contradiction.\newline Observe that each
element $(a, m)\in(R-\{0\})\propto M$ is regular fusible provided that $R$ is
regular fusible and $M$ is torsion free. Since $a$ is a nonzero element, it
can be written as $ar=z+t$ for some zero divisor $z$ and regular elements $r$
and $t$. Then $(a, m)(r, 0)=(z,0)+(t, mr)$ is a fusible decomposition, since
by Lemma \ref{lem2}, $(r, 0)$ and $(t, mr)$ are regular elements and $(z, 0)$
is a zero divisor.

\begin{theorem}
\label{division} Let $R$ be a right duo ring. Then every $R$-module is fusible
if and only if $R$ is a division ring.
\end{theorem}

\begin{proof}
Suppose that $R$ is a right duo ring. Then by assumption, $R$ is a fusible
$R$-module. Then for each $0\neq a\in R$ there exist $x\in T_{R}(R)$ and
$y\notin T_{R}(R)$ such that $a=x+y$. By the definition of fusible modules,
$x$ is a left zero divisor and $y$ is a non left zero divisor of $R$. Thus $R$
is a left fusible ring. Then by \cite[Lemma 2.19]{Ghas}, $R$ is a reduced
ring. Choose $0\neq a\in R$. Since $R$ is reduced, $a\notin ann_{R}(a)$. Now,
put $I=ann_{R}(a)$ and consider $R$-module $R/I$. Since $0_{R/I}\neq a+I\in
R/I$, there exist $x+I\in T_{R}(R/I)$ and $y+I\in T_{R}^{\star}(R/I)$ such
that $a+I=x+I+y+I$. Since $y+I\in T_{R}^{\star}(R/I)$, we have $ann_{R}%
(y+I)=0$. This means $(y+I)r=yr+I=I$ implies that $r=0$. Now we will show that
$ann_{R}(a)=0$. Let $t\in ann_{R}(a)=I$. Then $at=0$. Since $R$ is a right duo
ring, $Rt\subseteq tR$ which implies that $ayt=aty^{\prime}$ for some
$y^{\prime}\in R$. Since $at=0$, we have $ayt=0$ which implies that
$(y+I)t=yt+I=I$ and so $t=0$. Thus we have $ann_{R}(a)=0$. Now, we will show
that $a\in R$ has a right inverse. If $a^{2}R=R$, then $1=a^{2}r$ for some
$r\in R$ and $ar$ is a right inverse of $a$. So assume that $a^{2}R\neq R$.
Now, we will show that $a^{2}R=aR$. Let $a^{2}R\neq aR$ and put $J=a^{2}R$.
Consider $R$-module $R/J$. Since $R/J$ is a fusible module and $0+J\neq a+J\in
R/J$, we can write $a+J=x+J+y+J$ for some $x+J\in T_{R}(R/J)$ and $y+J\in
T_{R}^{\star}(R/J)$. As $x+J\in T_{R}(R/J)$, there exists $0\neq z\in R$ such
that $xz\in a^{2}R$. On the other hand, since $ann_{R}(y+J)=0$, $yt\in J$
implies that $t=0$. As $a-x-y\in a^{2}R$, we have $(a-x-y)az=a^{2}z-xaz-yaz\in
a^{2}R$. As $R$ is a right duo ring, we can write $xaz=xzt^{\prime}$ for some
$t^{\prime}\in R$. Which implies that $xaz=xzt^{\prime}\in a^{2}R$. As
$a^{2}z-xaz-yaz\in a^{2}R$, we conclude that $yaz\in a^{2}R$. Thus we have
$az=0$. Since $ann_{R}(a)=0$, we get $z=0$ which is a contradiction. Thus we
have $a^{2}R=aR$. Then we can write $a=a^{2}b$ for some $b\in R$. This implies
that $a(1-ab)=0$. As $ann_{R}(a)=0$, we get $1=ab$, that is, $a$ has a right
inverse. Hence, $R$ is a division ring. If $R$ is a division ring, then every
$R$-module is clearly a fusible module.
\end{proof}

\begin{corollary}
Let $R$ be a commutative ring. Then every $R$-module is fusible if and only if
$R$ is a field.
\end{corollary}

\end{document}